\documentclass{amsart}

\usepackage{tikz}

\usetikzlibrary{arrows,positioning,shapes.geometric,shapes.misc,calc,decorations.pathmorphing}

\usepackage{array}

\usepackage{graphicx}
\usepackage{verbatim}
\usepackage{amsthm}
\usepackage{amssymb}
\usepackage{amsmath}

\theoremstyle{remark}

\theoremstyle{definition}
\newtheorem{definition}{\bf Definition}[section]

\theoremstyle{plain}
\newtheorem{theorem}[definition]{Theorem}
\newtheorem{lemma}[definition]{Lemma}
\newtheorem{corollary}[definition]{Corollary}
\newtheorem{proposition}[definition]{Proposition}

\theoremstyle{remark}

\newtheorem{examples}[definition]{\bf Examples}
\newtheorem{problems}[definition]{\bf Problems}
\newtheorem{remark}[definition]{Remark}

\newcommand{\NA}{\mathbb{N}}

\newcommand{\cont}{\subseteq}

\newcommand{\set}[1]{\left\{#1\right\}}

\newcommand{\jn}[1]{\overline{#1}}

\newcommand{\dom}[1]{\mathop{dom}(#1)}

\begin{document}

\title{$n$--Arc Connected Spaces}
\author{Benjamin Espinoza, Paul Gartside, Ana Mamatelashvili.}
\address[Benjamin Espinoza]{Department of Mathematics\\ University of Pittsburgh at Greensburg\\
236 Frank A. Cassell Hall\\
150 Finoli Drive\\
Greensburg, PA 15601\\ USA}
\email{bee1@pitt.edu}
\address[Paul Gartside]{Department of Mathematics\\ University of Pittsburgh\\
508 Thackeray Hall\\
Pittsburgh, PA 15260\\ USA}
\email{gartside@math.pitt.edu}
\address[Ana Mamatelashvili]{Department of Mathematics\\ University of Pittsburgh\\
301 Thackeray Hall\\
Pittsburgh, PA 15260\\ USA}
\email{anm137@pitt.edu@math.pitt.edu}



\begin{abstract} A space is {\it $n$--arc connected} ($n$--ac) if any family of no more than $n$--points are contained in an arc. For graphs the following are equivalent: (i) $7$--ac, (ii) $n$--ac for all $n$, (iii) continuous injective image of a closed sub--interval of the real line, and (iv) one of a finite family of graphs. General continua that are $\aleph_0$--ac are characterized. The complexity of characterizing $n$--ac graphs for $n=2,3,4,5$ is determined to be strictly higher than that of the stated characterization of $7$--ac graphs. 
\end{abstract}

\maketitle

\section{Introduction}

A topological space $X$ is called {\it $n$--arc connected} ($n$--ac) if for any points $p_1, p_2, \dots, p_n$ in $X$, there
exists an arc $\alpha$ in $X$ such that $p_1, p_2, \dots p_n$ are all in $\alpha$. If a space is $n$--ac for all $n\in \NA$, then we will say that it is {\it $\omega$--ac}. Note that this is equivalent to saying that for any finite $F$ contained in $X$ there is an arc $\alpha$ in $X$ containing $F$. Call a space $\aleph_0$--ac if for every countable subset, $S$, there is an arc containing $S$.  Evidently a space is arc connected if and only if it is $2$--ac, and `$\aleph_0$--ac' implies `$\omega$--ac' implies `$(n+1)$--ac' implies `$n$--ac' (for any fixed $n$). 

Thus we have a family of natural strengthenings of arc connectedness, and the main aim of this paper is to characterize when `nice' spaces have one of these strong arc connectedness properties. Secondary aims are to distinguish `$n$--ac' (for each $n$), `$\omega$--ac' and `$\aleph_0$--ac', and to compare and contrast the familiar arc connectedness (i.e. $2$--ac) with its strengthenings.

Observe that any Hausdorff image of an $n$--ac (respectively, $\omega$--ac, $\aleph_0$--ac) space under a continuous injective map is also $n$--ac (respectively, $\omega$--ac, $\aleph_0$--ac). 
Below, unless explicitly stated otherwise, all spaces are (metrizable) continua. 

It turns out that  `sufficiently large' (in terms of dimension) arc connected spaces tend to be $\omega$--ac. Indeed it is not hard to see that manifolds (with or without boundary) of dimension at least $2$ are $\omega$--ac. Thus we focus on curves ($1$--dimensional continua) and especially on graphs (those connected spaces obtained by taking a finite family of arcs and then identifying some of the endpoints).

To motivate our main results consider the following examples.

\newpage

\begin{examples}{\label{exs}}  \ 

\begin{itemize}
\item[(A)] The arc (the closed unit interval, $I=[0,1]$) is $\aleph_0$--ac.

\item[(B)] The  open interval, $(0,1)$; and ray, $[0,1)$, are $\omega$--ac.

\item[(C)] From (A) and (B), all continua which are the continuous injective images of the arc, open interval and ray are $\omega$--ac. It is easy to verify that these include: (a) the arc, (b) the circle, (c)  figure eight curve, (d) lollipop, (e) dumbbell and (f) theta curve.

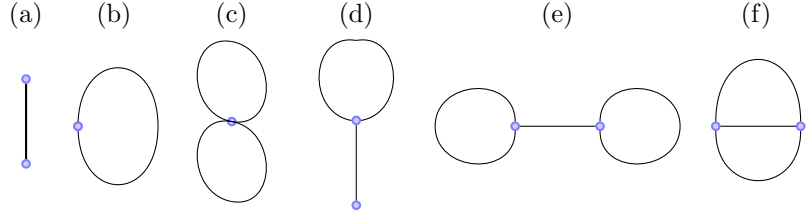
\begin{tabular}{cccccc}
(a) & (b) & (c) & (d) & (e) & (f) \\
\begin{tikzpicture}
 [vertex/.style={circle,draw=blue!50,fill=blue!20,thick, inner sep=0mm, minimum size=1mm},
  point/.style={coordinate}]

\node at (0,0) {};
\node  at (0,2) {};

\node[vertex] (arc_bot) at (0,.5) {};
\node[vertex] (arc_top) [above=of arc_bot] {}; 
\draw [thick] (arc_bot) -- (arc_top);
\end{tikzpicture}
&
\begin{tikzpicture}
 [vertex/.style={circle,draw=blue!50,fill=blue!20,thick, inner sep=0mm, minimum size=1mm},
  point/.style={coordinate}]
\node at (0,0) {};
\node  at (0,2) {};

\node[vertex] at (0,1) (circle_pt)  {};
\node[point] (pt)  [right=of circle_pt] {};

\draw  (circle_pt.center) to [out=90, in=90,looseness=2.5] (pt.center);
\draw   (pt.center) to [out=270, in=270, looseness=2.5] (circle_pt.center);

\node[vertex] at (circle_pt) {};


\end{tikzpicture}
&
\begin{tikzpicture}
 [vertex/.style={circle,draw=blue!50,fill=blue!20,thick, inner sep=0mm, minimum size=1mm},
  point/.style={coordinate}]

\node[vertex] (center_pt)  {};
\node[point] (top_pt)  [above=of center_pt] {};
\node[point] (bot_pt)  [below=of center_pt] {};

\draw  (center_pt.center) to [out=-10, in=-10,looseness=1.5] (top_pt);
\draw   (top_pt.center) to [out=170, in=170, looseness=1.5] (center_pt.center);

\draw  (center_pt.center) to [out=-10, in=-10,looseness=1.5] (bot_pt);
\draw   (bot_pt.center) to [out=170, in=170, looseness=1.5] (center_pt.center);



\end{tikzpicture}
&
\begin{tikzpicture} [vertex/.style={circle,draw=blue!50,fill=blue!20,thick, inner sep=0mm, minimum size=1mm}, point/.style={coordinate}]

\node[vertex] (lollipop_bot) {};
\node[vertex] (lollipop_mid) [above=of lollipop_bot] {}; 
\node[point] (pt) [above=of lollipop_mid] {};

\draw  (lollipop_bot.north)  -- (lollipop_mid.south);
\draw  (lollipop_mid.east) to [out=10,in=10,looseness=1.5] (pt);
\draw   (pt) to [out=170, in=170,looseness=1.5] (lollipop_mid.west);
\end{tikzpicture}
&
\begin{tikzpicture}
 [vertex/.style={circle,draw=blue!50,fill=blue!20,thick, inner sep=0mm, minimum size=1mm},
  point/.style={coordinate}]
\node at (0,0) {};
\node  at (0,2) {};

\node[vertex] at (0,1) (dumbbell_left) {};
\node[vertex] (dumbbell_right) [right=of dumbbell_left] {}; 
\node[point] (pt1) [left=of dumbbell_left] {};
\node[point] (pt2) [right=of dumbbell_right] {};

\draw  (dumbbell_left.east)  -- (dumbbell_right.west);
\draw  (dumbbell_left.north) to [out=90, in=90,looseness=1.5] (pt1.north);
\draw   (pt1.south) to [out=270, in=270,looseness=1.5] (dumbbell_left.south);
\draw  (dumbbell_right.north) to [out=90, in=90,looseness=1.5] (pt2.north);
\draw   (pt2.south) to [out=270, in=270,looseness=1.5] (dumbbell_right.south);
\end{tikzpicture}
&
\begin{tikzpicture}
 [vertex/.style={circle,draw=blue!50,fill=blue!20,thick, inner sep=0mm, minimum size=1mm},
  point/.style={coordinate}]

\node at (0,0) {};
\node  at (0,2) {};

\node[vertex] at (0,1) (theta_left) {};
\node[vertex] (theta_right) [right=of theta_left] {};

\draw (theta_left.east) to (theta_right.west) {};
\draw  (theta_left.north) to [out=90, in=90,looseness=2.5] (theta_right.north);
\draw (theta_right.south) to [out=270, in=270,looseness=2] (theta_left.south);
\end{tikzpicture}
\end{tabular}

\item[(D)] The Warsaw circle; double Warsaw circle; Menger cube; and Sierpinski triangle, are $\omega$--ac.

\item[(E)] The simple triod  is $2$--ac but not $3$--ac.
It is minimal in the sense that no graph with strictly fewer edges is $2$--ac not $3$--ac.

The graphs (a), (b) and~(c)  below are: $3$--ac but not $4$--ac, $4$--ac but not $5$--ac, and
$5$--ac but not $6$--ac, respectively. All are minimal.
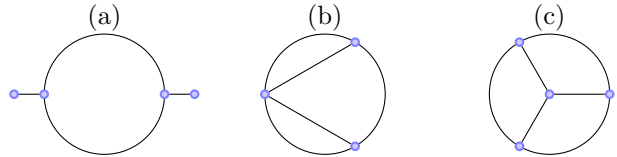
\begin{center}
\begin{tabular}{ccc}
(a) & (b) & (c)  \\

\begin{tikzpicture}[vertex/.style={circle,draw=blue!50,fill=blue!20,thick, inner sep=0mm, minimum size=1mm},
  point/.style={coordinate}]

\draw (0,0) ellipse (8mm and 8mm);
\node[vertex] (v1) at (180:12mm) {};
\node[vertex] (v2) at (0:12mm) {};
\node[vertex] (v3) at (180:8mm) {};
\node[vertex] (v4) at (0:8mm) {};

\draw (v1) -- (v3);
\draw (v2) -- (v4);
\end{tikzpicture}
&
\begin{tikzpicture}[vertex/.style={circle,draw=blue!50,fill=blue!20,thick, inner sep=0mm, minimum size=1mm},
  point/.style={coordinate}]
\node at (-12mm,0) {};
\node  at (12mm,0) {};

\draw (0,0) circle (8mm);
\node[vertex] (v1) at (60:8mm) {};
\node[vertex] (v2) at (180:8mm) {};
\node[vertex] (v3) at (300:8mm) {};
\draw (v1) -- (v2) -- (v3);

\end{tikzpicture}
&
\begin{tikzpicture}[vertex/.style={circle,draw=blue!50,fill=blue!20,thick, inner sep=0mm, minimum size=1mm},
  point/.style={coordinate}]

\node at (-12mm,0) {};
\node  at (12mm,0) {};

\draw (0,0) circle (8mm);

\node[vertex] (v1) at (0:8mm) {};
\node[vertex] (v2) at (120:8mm) {};
\node[vertex] (v3) at (240:8mm) {};

\draw (0,0) -- (v1);
\draw (0,0) -- (v2);
\draw (0,0) -- (v3);

\node[vertex] (v0) at (0,0) {};
\end{tikzpicture}
\end{tabular}
\end{center}


\item[(F)] The Kuratowski graph $K_{3,3}$ is $6$--ac but not $7$--ac. It is also minimal.

\item[(G)] The graphs below are all $6$--ac and,
by Theorem~\ref{main1},  none is $7$--ac. Unlike  $K_{3,3}$ all are planar. It is 
unknown if the first of these graphs (which has $12$ edges) is minimal among planar graphs. A minimal example must have  at 
least nine edges.

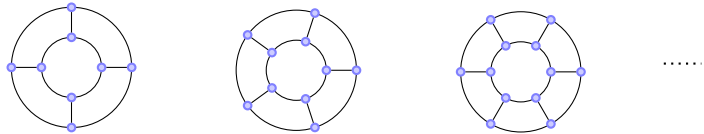
\begin{center}
\begin{tabular}{cccc}
\begin{tikzpicture}[vertex/.style={circle,draw=blue!50,fill=blue!20,thick, inner sep=0mm, minimum size=1mm},
  point/.style={coordinate}]
\node at (-12mm,0) {};
\node  at (12mm,0) {};

\draw (0,0) circle (8mm);
\draw (0,0) circle (4mm);
\node[vertex] (v1) at (0:8mm) {};
\node[vertex] (v2) at (90:8mm) {};
\node[vertex] (v3) at (180:8mm) {};
\node[vertex] (v4) at (270:8mm) {};
\node[vertex] (v1i) at (0:4mm) {};
\node[vertex] (v2i) at (90:4mm) {};
\node[vertex] (v3i) at (180:4mm) {};
\node[vertex] (v4i) at (270:4mm) {};

\draw (v1i) -- (v1);
\draw (v3i) -- (v3);
\draw (v2i) -- (v2);
\draw (v4i) -- (v4);
\end{tikzpicture}
&
\begin{tikzpicture}[vertex/.style={circle,draw=blue!50,fill=blue!20,thick, inner sep=0mm, minimum size=1mm},
  point/.style={coordinate}]
\node at (-12mm,0) {};
\node  at (12mm,0) {};

\draw (0,0) circle (8mm);
\draw (0,0) circle (4mm);
\node[vertex] (v1) at (0:8mm) {};
\node[vertex] (v2) at (72:8mm) {};
\node[vertex] (v3) at (144:8mm) {};
\node[vertex] (v4) at (216:8mm) {};
\node[vertex] (v5) at (288:8mm) {};
\node[vertex] (v1i) at (0:4mm) {};
\node[vertex] (v2i) at (72:4mm) {};
\node[vertex] (v3i) at (144:4mm) {};
\node[vertex] (v4i) at (216:4mm) {};
\node[vertex] (v5i) at (288:4mm) {};

\draw (v1i) -- (v1);
\draw (v3i) -- (v3);
\draw (v2i) -- (v2);
\draw (v4i) -- (v4);
\draw (v5i) -- (v5);
\end{tikzpicture}
&
\begin{tikzpicture}[vertex/.style={circle,draw=blue!50,fill=blue!20,thick, inner sep=0mm, minimum size=1mm},
  point/.style={coordinate}]
\node at (-12mm,0) {};
\node  at (12mm,0) {};

\draw (0,0) circle (8mm);
\draw (0,0) circle (4mm);
\node[vertex] (v1) at (0:8mm) {};
\node[vertex] (v2) at (60:8mm) {};
\node[vertex] (v3) at (120:8mm) {};
\node[vertex] (v4) at (180:8mm) {};
\node[vertex] (v5) at (240:8mm) {};
\node[vertex] (v6) at (300:8mm) {};
\node[vertex] (v1i) at (0:4mm) {};
\node[vertex] (v2i) at (60:4mm) {};
\node[vertex] (v3i) at (120:4mm) {};
\node[vertex] (v4i) at (180:4mm) {};
\node[vertex] (v5i) at (240:4mm) {};
\node[vertex] (v6i) at (300:4mm) {};

\draw (v1i) -- (v1);
\draw (v3i) -- (v3);
\draw (v2i) -- (v2);
\draw (v4i) -- (v4);
\draw (v5i) -- (v5);
\draw (v6i) -- (v6);
\end{tikzpicture}
&
{
\begin{tikzpicture}
\node at (0,8mm) {};
\draw[thick, dotted]  (0,0) -- (0.5,0);
\node at (0,-8mm) {};
\end{tikzpicture}
}
\end{tabular}
\end{center}
\end{itemize}
\end{examples}

The diversity of examples in (D) of $\omega$--ac curves suggests that no simple characterization of these continua is likely. The authors, together with Kovan--Bakan, prove that there is indeed no characterization of $\omega$--ac curves any simpler than the definition, see \cite{egkm} for details.

This prompts us to restrict attention to the more concrete case of graphs, and leads us to the following natural problems.
\begin{problems} \ 

\begin{itemize}
\item[(1)] Characterize the $\omega$--ac graphs.
\item[(2)] Characterize the $\aleph_0$--ac graphs.
\item[(3)] Characterize, for each $n$, the graphs which are $n$--ac but not $(n+1)$--ac.
\end{itemize}
\end{problems}

In Section~\ref{omegag} below we show that the list of $\omega$--ac graphs given in (C) is complete, answering Problem~(1). 
\begin{theorem}{\label{main1}} 
For a graph $G$ the following are equivalent:

\begin{itemize}
\item[1)] $G$ is $7$--ac,
\item[2)] $G$ is $\omega$--ac,
\item[3)] $G$ is the continuous injective image of a sub--interval of the real line,
\item[4)] $G$ is one of the following graphs: the arc, simple closed curve, figure eight curve, lollipop, dumbbell or theta curve.
\end{itemize}
\end{theorem}
In Section~\ref{aleph0} we go on to characterize the $\aleph_0$--ac continua, giving a very strong solution to Problem~(2). 
\begin{theorem} {\label{main2}} 
For any continuum $K$ (not necessarily metrizable) the following are equivalent:
\begin{itemize}
\item[1)] $K$ is $\aleph_0$--ac,
\item[2)] $K$ is a continuous injective image of a closed sub--interval of the long line,
\item[3)] $K$ is one of: the arc, the long circle, the long lollipop, the long dumbbell, the long figure eight or the long theta-curve.
\end{itemize}
\end{theorem}

From Theorem~\ref{main1} we see that there are no examples of graphs that are $n$--ac but not $(n+1)$--ac, for $n \ge 7$, solving Problem~(3) in these cases. Of course Examples~(E)  and~(F)  show that there are $n$--ac not $(n+1)$--ac graphs for $n=2,3,4,5$ and $6$. But the question remains: can we characterize these latter graphs? Informally our answers are as follows.
\begin{theorem} {\label{main3}} \ 

\begin{itemize}
\item[1)] The characterization of $\omega$--ac graphs given in Theorem~\ref{main1} is as simple as possible,
\item[2)] there exist reasonably simple characterizations of $n$--ac not $(n+1)$--ac graphs for $n \le 7$, 
\item[3)] for $n=2,3,4$ and $5$, there is no possible characterization of $n$--ac not $(n+1)$--ac graphs which is as simple as that for $\omega$--ac graphs.
\end{itemize}
\end{theorem}
In Section~\ref{cplxty} we outline some machinery from descriptive set theory which allows us to formalize and prove these claims.
The situation with $n=6$ --- the complexity of characterizing  graphs which are $6$--ac but not $7$--ac --- remains unclear. This, and other remaining open problems, are discussed in Section~\ref{probs}.

\section{Characterizations}
In this section we prove the characterization theorems stated in the Introduction. First Theorem~\ref{main1} characterizing $\omega$--ac graphs. Second Theorem~\ref{main2} characterizing $\aleph_0$--ac continua. 

\subsection{$\omega$--ac Graphs}\label{omegag}

As noted in Example~(C) the graphs listed in part 4) of Theorem~\ref{main1} are all the continuous injective image of a closed sub--interval of the real line, giving 4) implies 3), and all such images are $\omega$--ac, yielding 3) implies 2) of Theorem~\ref{main1}. Clearly $\omega$--ac graphs are $7$--ac, and so 2) implies 1) in Theorem~\ref{main1}. 

It remains to show 1) implies 4) in Theorem~\ref{main1}, in other words that any $7$--ac graph is one of the graphs listed in 4). This is established in Theorem~\ref{ins} below. We proceed by establishing an ever tightening sequence of restrictions on the structure of $7$--ac graphs.

We note that Lelek and McAuley, \cite{lelek}, showed that the only Peano continua which are continuous injective images of the real line are the figure eight, dumbbell and theta--curve. Their proof can be modified to establish the equivalence of 3) and 4) in Theorem~\ref{main1}. 

\begin{proposition}
Let $G$ be a finite graph, and let $H \subset G$ be a subgraph of $G$ such that 
$G-H$ is connected and $\overline{G-H}\cap H = r$ is a branch point of $G$.
If $G$ is $n$--ac, then $\overline{G-H}$ is $n$--ac.
\label{complement}
\end{proposition}

\begin{proof}
First note that $\overline{G-H} = (G-H)\cup \set{r}$. Hence every connected set intersecting $G-H$ and $H-\set{r}$,
must contain $r$. 

Let $\mathcal{P}=\set{p_1, p_2, \dots , p_n}$ be a set of $n$ points in 
$\overline{G-H}$. Then, since $G$ is $n$--ac, there exists an arc $\alpha$ in $G$ containing $\mathcal{P}$.
If $\alpha \subset \overline{G-H}$, we are done. So assume $\alpha$
intersects $H-\set{r}$. Let $t_0, t_1\in [0,1]$ such that $\alpha(t_0)\in G-H$
and $\alpha(t_1)\in H-\set{r}$, assume without loss of generality that $t_0< t_1$. Hence
there exists $s\in [t_0, t_1]$ such that $\alpha(s)=r$. Then $\alpha([0,s])$ is an arc in $\overline{G-H}$
containing $\mathcal{P}$, otherwise $r\in \alpha((s,1])$ which is impossible since $\alpha$ is an injective image of $[0,1]$.
This proves that $\overline{G-H}$ is $n$--ac.
\end{proof}

The reverse implication of Proposition~\ref{complement} does not hold. To see this, let $G$ be a simple triod and
$H$ be one of the edges of $G$. Clearly $G$ is not $3$--ac but $\overline{G-H}$ (an arc) is $3$--ac.

\begin{definition}
Let $G$ be a finite graph. An edge $e$ of $G$ is called a {\it terminal edge} of $G$ if 
one of the vertices of $e$ is an end--point of $G$.
\label{isolated}
\end{definition}

\begin{definition}
Let $G$ be a finite graph, and let $I=\set{e_1, e_2, \dots e_m}$ be the set of terminal edges of $G$. 
Let $G^*$ be the graph given by $\overline{G-I}$. Clearly this operation can be applied to $G^*$
as well. We perform this operation as many times as necessary until we obtain a graph
$G'$ having no terminal edges. We called the graph $G'$ the {\it reduced graph of} $G$.
\label{reduced}
\end{definition}

The following is a corollary of Proposition~\ref{complement}.

\begin{corollary}
Let $G$ be an $n$--ac finite graph. Then the reduced graph of $G$ is an $n$--ac finite graph containing
no terminal edges.
\label{nwacreduced}
\end{corollary}

\begin{proof}
Observe that the reduced graph of $G$ can also be obtained by removing terminal edges one at a time.

Now, from Proposition~\ref{complement}, if $G$ is an $n$--ac finite graph and $e$ is a terminal 
edge of $G$, then $\overline{G-e}$ is $n$--ac. This implies that each time we remove a terminal edge
we obtain an $n$--ac graph. This and the observation prove the corollary.
\end{proof}

\begin{remark}
Note that if $X$ is an $n$--ac space and $\set{p_1, p_2, \dots , p_n}$ are $n$ different points of $X$, then
there is an arc $\alpha$ such that $\set{p_1, p_2, \dots , p_n}\subset \alpha$ and such that the
end--points of $\alpha$ belong to $\set{p_1, p_2, \dots , p_n}$. To see this, let $\beta$ be the arc
containing $\set{p_1, p_2, \dots , p_n}$, given by the fact that $X$ is $n$--ac. 
Let $t_0=\text{min}\set{\beta^{-1}(p_i) \, |\,  i=1, \dots , n}$ and 
$t_1=\text{max}\set{\beta^{-1}(p_i) \, | \, i=1, \dots ,n}$. Then $\beta([t_0, t_1])$ satisfies the
conditions of $\alpha$.

From now on, if $X$ is an $n$--ac space, $\set{p_1, \dots , p_n}$ are $n$ different points and
$\alpha$ is an arc passing through $\set{p_1, p_2, \dots , p_n}$, then we will assume that the
end--points of $\alpha$ belong to $\set{p_1, p_2, \dots , p_n}$.
\label{endpoints}
\end{remark}

\begin{lemma}
Let $G$ be a finite graph. Assume that $G$ contains a simple triod $T=L_1\cup L_2\cup L_3$ (with
$\set{q}=L_i\cap L_j$, for $i\neq j$) such that for each $i$, $L_i-\set{q}$
contains no branch points of $G$. For each $i=1, 2, 3$, let $p_i\in \text{int}(L_i)$.
If $\alpha$ is an arc containing $\set{p_1, p_2, p_3}$, then

\begin{enumerate}
\item $q\in \text{int}(\alpha)$, and
\item at least one of the end points of $\alpha$ lies in $[q,p_1]\cup [q,p_2]\cup [q,p_3]$.
\end{enumerate}
\label{4ptsarc}
\end{lemma}

\begin{proof}
Let $G$, $T$ and $p_1, p_2, p_3$ as in the hypothesis of the lemma.
Let $\alpha \subset G$ be an arc containing $\set{p_1, p_2, p_3}$, and
denote, for each $i=1,2,3$, by $[q, l_i]$ the arc $L_i$.

\begin{enumerate}
\item Assume, without loss of generality, that $\alpha(t_i)=p_i$ and that
$t_1 < t_2 <t_3$. Then $p_2\in \text{int}(\alpha)$ and $\alpha = \alpha([0, t_2])\cup \alpha([t_2, 1])$.

We consider two cases: $q\not\in \alpha([0,t_2])$ and $q\in \alpha([0,t_2])$.
Assume $q\not\in \alpha([0,t_2])$, then, since $L_2-\set{q}$ contains no branch points of $G$ and
$p_2 \in \text{int}(L_2)$, we have that $l_2=\alpha(s)$ for some $s$ with $0<s<t_2$. Hence $[l_2,p_2]\cont \alpha([0,t_2])$.
Therefore, since $\set{p_2, p_3}\cont \alpha([t_2,1])$, $p_2\in \text{int}(L_2)$, $L_2-q$ has no
branch points of $G$, and $\alpha$ is a $1-1$ function, we have that $[p_2,q]\cont \alpha([t_2, 1))$.
This implies that $q\in \text{int}(\alpha)$.

Now suppose that $q\in \alpha([0,t_2])$. If $q\in \alpha((0,t_2])$, then we are done. So assume that
$q=\alpha(0)$, i.e. $q$ is an end-point of $\alpha$. Using the same argument as in the previous case, 
we can conclude that $[l_2,p_2]\cont \alpha([0,t_2])$. This implies, as before, that $[p_2,q]\cont \alpha([t_2, 1))$
wich contradicts the fact that $\alpha$ is a $1-1$ function. Hence $q\in \text{int}(\alpha)$.

\item First, assume that $\alpha(t_i)=p_i$ and that $t_1 < t_2 < t_3$. We
will show that one end point of $\alpha$ lies on either $[q,p_1]$ or $[q,p_3]$. The other cases (rearrangements of the $t_i$s)
are done in the same way as this case, the only difference is the conclusion: the end point lies
either on $[q,p_1]$ or $[q,p_2]$, or the end point lies either on $[q,p_2]$ or $[q,p_3]$.

By (1), $q\in \text{int}(\alpha)$ and if $q=\alpha(s)$, then $s<t_3$; otherwise 
the arc $\alpha([0,t_3])$
would contain $p_1,p_2,p_3$ and $q\not\in\text{int}(\alpha([0,t_3]))$ which is contrary to (1).
Similarly, $t_1<s$. Hence $t_1<s<t_3$.

If $s<t_2$, then $p_1, q\not\in \alpha([t_2,1])=\alpha([t_2,t_3])\cup \alpha([t_3, 1])$.
Now, since $L_3-\set{q}$ has no branch points of $G$, $q\in \alpha([0,t_2])$, and 
$p_3\in\text{int}(L_3)$, we have $l_3\in \alpha([t_2,t_3])$. Thus, since $\alpha$
is a $1-1$ function, $\alpha([t_3,1])\subset (q,p_3]$. This shows that $\alpha(1)$ lies in $[q,p_3]$.

If $t_2<s$, then a similar argument using $-\alpha$ ($\alpha$ traveled in the opposite direction)
shows that one of the end points of $\alpha$ lies on $[q,p_1]$.
\end{enumerate}
\end{proof}

We obtain the following corollaries.

\begin{corollary}
With the same conditions as in Lemma~\ref{4ptsarc}. If $\alpha$ is an arc containing $\set{p_1,p_2,p_3}$, and
$q=\alpha(s)$, $p_i=\alpha(t_i)$ for $i=1,2,3$, then $t_j<s <t_k$ for some $j,k\in\set{1,2,3}$.
\label{inbetween}
\end{corollary}

\begin{proof}
To see this, note that if $q$ does not lie between two of the $p_i$s, then either $s<t_i$ for all $i$, or
$t_i<s$ for all $i$. Then either $\alpha([s,1])$ or $\alpha([0,s])$ are arcs containing $\set{p_1,p_2,p_3}$ for
which $q$ is an end-point, this contradicts (1) of Lemma~\ref{4ptsarc}.
\end{proof}

\begin{corollary}
Let $G$ be a finite graph, and let $\set{p_1, p_2, \dots p_n}\subset G$ be $n$ different
points. In addition, let $\alpha$ be one arc containing $\set{p_1, p_2, \dots p_n}$, with end--points belonging
to $\set{p_1, p_2, \dots p_n}$. If there are three different indexes $i,j,k$ such that $p_i$, $p_j$ and $p_k$ belong 
to a triod $T$ satisfying the conditions of (Lemma~\ref{4ptsarc}), and such that 
$([q,p_i]\cup [q,p_j]\cup [q,p_k])\cap \set{p_1, p_2, \dots p_n} =\set{p_i, p_j, p_k}$, then either
$p_i$, $p_j$ or $p_k$ is an end--point of $\alpha$.
\label{triodendpt}
\end{corollary}

\begin{proof}
By (2) of Lemma~\ref{4ptsarc}, at least one of the end points of $\alpha$ lies in $[q,p_i]\cup [q,p_j]\cup [q,p_k]$.
Hence, since the end-points of $\alpha$ belong to $\set{p_1, p_2, \dots p_n}$ and
$([q,p_i]\cup [q,p_j]\cup [q,p_k])\cap \set{p_1, p_2, \dots p_n} =\set{p_i, p_j, p_k}$, one of $p_i$, $p_j$ or $p_k$ 
is an end-point of $\alpha$
\end{proof}

\begin{proposition}
Let $G$ be a finite graph. If $G$ is $5$--ac, then $G$ has no branch point of
degree greater than or equal to five.
\label{degree5}
\end{proposition}

\begin{proof}
Assume, by contradiction, that $G$ contains at least one branch point, $q$, of degree at least $5$.
Then, since $G$ is a finite graph, $G$ contains a simple $5$-od, $T=L_1\cup L_2 \cup L_3\cup L_4\cup L_5$, such that
$\set{q}=L_i\cap L_j$ for $i\neq j$, and such that $L_i-\set{q}$ contains
no branch points of $G$.

For each $i=1, \dots, 5$, let $p_i\in \text{int}(L_i)$. Then, since $G$ is $5$--ac, there exists an
arc $\alpha \subset G$ such that $\set{p_1, p_2, \dots, p_5}\subset \alpha$. Note that $T$ contains a triod
satisfying the conditions of Lemma~\ref{4ptsarc}, hence
$q\in \text{int}(\alpha)$. Let $t_0\in (0,1)$ be the point such that $\alpha(t_0)=q$. Then
$\alpha-\set{q}=\alpha([0,t_0))\cup \alpha((t_0,1])$, and either $\alpha([0,t_0])$ or $\alpha([t_0,1])$
contains three points out of $\set{p_1, p_2, p_3, p_4, p_5}$, note that $q$ is an end--point of
$\alpha([0,t_0])$ and of $\alpha([t_0,1])$. 
Without loss of generality, suppose that $p_1, p_2, p_3 \subset \alpha([0,t_0])$; then $L_1, L_2, L_3$
and the corresponding $p_i$s satisfy the conditions of Lemma~\ref{4ptsarc} implying that any arc 
containing those points contains $q$ in its interior, a contradiction, since $q$ is an
end point of $\alpha([0,t_0])$. This shows that $G$
does not contain a branch point of degree greater than or equal to five.
\end{proof}

From Proposition~\ref{degree5} we obtain the following corollaries.

\begin{corollary}
Let $G$ be a finite graph. If $G$ is $n$--ac, for $n\geq 5$, then $G$ has no branch point of
degree greater than or equal to five.
\label{nodegree5n}
\end{corollary}

The following proposition is easy to prove.

\begin{proposition}
Let $G$ be a finite connected graph. If $G$ has at least three branch points, then there is an arc
$\alpha$ such that the end-points of $\alpha$ are branch points of $G$ and all the points of
the interior of $\alpha$, except for one, are not branch points of $G$. So $\alpha$ contains
exactly three branch points of $G$.
\label{p3branchpts}
\end{proposition}

\begin{theorem}
A finite graph with three or more branch points cannot be $7$--ac.
\label{no7wac}
\end{theorem}

\begin{proof} Let $G$ be a finite graph with at least three branch points.

By Proposition~\ref{p3branchpts}, there is an arc $\alpha$ in $G$ containing exactly three branch points of $G$ such
that two of them are the end--points of $\alpha$. Denote by $q_1$, $q_2$, and $q_3$ these branch points, and
assume without loss of generality that $q_1$ and $q_3$ are the end-points of $\alpha$.

Let $p_3$ be a point between $q_1$ and $q_2$, and let $p_5$ be a point between $q_2$ and $q_3$.
Since $G$ is a finite graph, we can find, in a neighborhood of $q_1$, two points $p_1$ and $p_2$ such
that $p_1$, $p_2$, $p_3$ belong to a triod $T_1$ satisfying the conditions of Lemma~\ref{4ptsarc}, and such that
$q_1$ is the branch point of $T_1$. Similarly, we can find a point $p_4$, in a neighborhood of $q_2$, such that
$p_3$, $p_4$ and $p_5$ belong to a triod $T_2$ satisfying the conditions of Lemma~\ref{4ptsarc}, and such that
$q_2$ is the branch point of $T_2$. Finally, we can find two points $p_6$, and $p_7$, in a neighborhood of $q_3$,
such that $p_5$, $p_6$ and $p_7$ belong to a triod $T_3$ satisfying the conditions of Lemma~\ref{4ptsarc}, and such that
$q_3$ is the branch point of $T_3$.

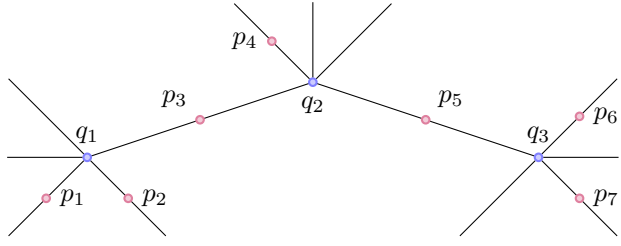
\begin{center}
\begin{tikzpicture}[vertexq/.style={circle,draw=blue!50,fill=blue!20,thick, inner sep=0mm, minimum size=1mm},
vertexp/.style={circle,draw=purple!50,fill=purple!20,thick, inner sep=0mm, minimum size=1mm},
  point/.style={coordinate}]

\node[point] (pq2)  {};
\node[point] (pq1) [below left=1cm and 3cm of pq2] {};
\node[point] (pq3) [below right=1cm and 3cm of pq2] {};

\draw (pq1) -- (pq2) node[vertexp,midway,label=above left:$p_3$] {};
\draw (pq2) -- (pq3) node[vertexp,midway,label=above right:$p_5$] {};

\node[vertexq] (q1) [label=above:$q_1$] at (pq1) {};
\node[vertexq] (q2) [label=below:$q_2$] at (pq2) {};
\node[vertexq] (q3) [label=above:$q_3$] at (pq3) {};

\node[point] (q21) [above left=of q2] {};
\node[point] (q22) [above=of q2] {};
\node[point] (q23) [above right=of q2] {};

\draw (q2) -- (q21) node[vertexp,midway,label=left:$p_4$] {};
\draw (q2) -- (q22);
\draw (q2) -- (q23);

\node[point] (q11) [left=of q1] {};
\node[point] (q12) [above left=of q1] {};
\node[point] (q13) [below left=of q1] {};
\node[point] (q14) [below right=of q1] {};

\draw (q1) -- (q11);
\draw (q1) -- (q12);
\draw (q1) -- (q13) node[vertexp,midway,label=right:$p_1$] {};
\draw (q1) -- (q14) node[vertexp,midway,label=right:$p_2$] {};

\node[point] (q31) [right=of q3] {};
\node[point] (q32) [above right=of q3] {};
\node[point] (q33) [below left=of q3] {};
\node[point] (q34) [below right=of q3] {};

\draw (q3) -- (q31);
\draw (q3) -- (q32) node[vertexp,midway,label=right:$p_6$] {};
\draw (q3) -- (q33);
\draw (q3) -- (q34) node[vertexp,midway,label=right:$p_7$] {};

\end{tikzpicture}
\end{center}


We show by contradiction that there is no arc containing $\set{p_1, p_2, \dots , p_7}$.
Suppose that there is an arc $\beta\subset G$ containing the points $\set{p_1, p_2, \dots , p_7}$,
using the same argument from Remark~\ref{endpoints}, we can assume that the end points of $\beta$ belong to 
$\set{p_1, p_2, \dots , p_7}$.

Now, by Corollary~\ref{triodendpt}, one of $\set{p_1, p_2, p_3}$ is an end point of $\beta$. Similarly, one of
$\set{p_3, p_4, p_5}$ is an end--point of $\beta$, and one of $\set{p_5, p_6, p_7}$ is an end--point of $\beta$. So,
since $\beta$ is an arc with end--points in $\set{p_1, p_2, p_3, \dots , p_7}$, we have that either
\begin{enumerate}
\item[(i)] $p_1$ or $p_2$ and $p_5$ are the end--points of $\beta$, or

\item[(ii)] $p_6$ or $p_7$ and $p_3$ are the end--points of $\beta$ or

\item[(iii)] $p_3$ and $p_5$ are the end--points of $\beta$,

\end{enumerate}

are the only possible cases. We will prove that every case leads to a contradiction.

\begin{enumerate}
\item[(i)] Assume that $p_1$ and $p_5$ are the end--points of $\beta$. Since the arc between
$q_2$ and $q_3$ contains no branch points of $G$, we have that either $[q_2, p_5]\subset \beta$ or
$[p_5, q_3]\subset \beta$. 

Assume first that $[q_2, p_5]\subset \beta$. Then, by the way $p_4$ was chosen
and the fact that $p_4 \in \text{int}(\beta)$, the arc $[p_4, q_2]\subset \beta$; similarly, since the arc
$[q_1,q_2]$ contains no branch points of $G$ and the fact that $p_3\in\text{int}(\beta)$, the arc
$[p_3,q_2]\subset\beta$. Then $\left([p_3,q_2] \cup [p_4, q_2] \cup [q_2, p_5]\right)\subset \beta$, which
is a contradiction since $\left([p_3,q_2] \cup [p_4, q_2] \cup [q_2, p_5]\right)$ is a nondegenerate
simple triod.

Assume that $[p_5, q_3]\subset \beta$. Then, by the way $p_6$ was chosen and the fact that $p_6\in\text{int}(\beta)$,
the arc $[q_3, p_6]\subset \beta$. Using the same argument we can conclude that the arc $[q_3, p_7]\subset \beta$.
Hence $\left([p_5,q_3] \cup [p_6, q_3] \cup [q_3, p_7]\right)\subset \beta$, which is a contradiction.

The case when $p_2$ and $p_5$ are the end--points of $\beta$ is similar the case we just proved. So
(i) does not hold.

\item[(ii)] This case is equivalent to (i), therefore (ii) does not hold.

\item[(iii)] Assume that $p_3$ and $p_5$ are the end-points of $\beta$. Then, since the arc
$[q_1,q_2]$ contains no branch points of $G$ and $p_3$ is an end-point of $\beta$, either
$[q_1, p_3]\subset \beta$ or $[p_3,q_2]\subset \beta$.

Suppose that $[q_1, p_3]\subset \beta$. As in ($i$), since $p_1, p_2 \in\text{int}(\beta)$, we have 
that the arc $[p_1,q_1]$ and $[q_1,p_2]$ are contained in $\beta$. Implying that the nondegenerate simple triod 
$\left([q_1, p_3] \cup [p_1,q_1] \cup [q_1,p_2]\right) \subset \beta$, which is a contradiction.

Now assume that $[p_3,q_2]\subset \beta$. Then the arc $[p_5,q_3]\subset \beta$. Again,
the same argument as in (i) leads to a nondegenerate simple triod being contained in $\beta$
since $p_6, p_7\in\text{int}(\beta)$. Hence (iii) does not hold.
\end{enumerate}

This proves that there is no arc containing $\set{p_1, p_2, \dots , p_7}$. Therefore $G$ is not
$7$--ac.
\end{proof}

Since every $(n+1)$--ac space is $n$--ac, we have the following corollary.

\begin{corollary}
A finite graph with three or more branch points cannot be a $n$--ac, for $n\geq 7$.
\end{corollary}

\begin{lemma}
If $G$ is a finite graph with only $2$ branch points each of degree greater than or equal to $4$, then
$G$ is not $7$--ac.
\end{lemma}

\begin{proof}
Let $q_1$ and $q_2$ be the two branch points of $G$. Then there exists at least one edge $e$ having
$q_1$ and $q_2$ as vertices. Let $p_1\in \text{int}(e)$. Since $G$ is a finite graph, and $q_1$ and $q_2$ have
degree at least 4, we can chose three points $p_2$, $p_3$, $p_4$ in a neighborhood of $q_1$ such that
$T_1=[q_1, p_1] \cup [q_1, p_2] \cup [q_1, p_3] \cup [q_1, p_4]$ is a simple $4$-od, and three points
$p_5$, $p_6$, $p_7$ in a neighborhood of $q_2$ such that $T_2=[q_2, p_1] \cup [q_2, p_5] \cup [q_2, p_6] \cup [q_2, p_7]$
is a simple $4$--od, and they are such that $T_1\cap T_2 =\set{p_1}$.

We show by contradiction, that there is no arc $\alpha\subset G$ containing $\set{p_1, p_2, \dots , p_7}$.
For this suppose that there exists such an arc $\alpha$, assume further that the end--points of $\alpha$ belong to
$\set{p_1, p_2, \dots , p_7}$. Then, since $\set{p_2, p_3, p_4}$ satisfy the conditions of Corollary~\ref{triodendpt}, we
can assume without loss of generality that $p_4$ is an end--point of $\alpha$. Similarly for the set $\set{p_5, p_6, p_7}$,
so we can assume without loss of generality that $p_5$ is an end--point of $\alpha$. On the other hand, the set
$\set{p_1, p_2, p_3}$ also satisfies the conditions of Corollary~\ref{triodendpt}, hence $p_1$, or $p_2$ or $p_3$ is an
end-point of $\alpha$ which is impossible since $\alpha$ only has two end-points. This shows that there is no
arc in $G$ containing $\set{p_1, p_2, \dots , p_7}$. This proves that $G$ is not $7$--ac.
\end{proof}

\begin{corollary}
If $G$ is a finite graph with only $2$ branch points each of degree greater than or equal to $4$, then
$G$ is not $n$--ac, for $n\geq 7$.
\label{no24}
\end{corollary}

\begin{theorem}
Let $G$ be a finite graph. If $G$ is $7$--ac, then $G$ is one of the following graphs: arc, simple closed curve, figure eight, lollipop, dumbbell or theta--curve.
\label{ins}
\end{theorem}

\begin{proof}
Let $G$ be a finite graph. Suppose that $G$ is $n$--ac, for $n\geq 7$. We will show that $G$ is (homeomorphic to) one of the listed graphs.

Let $K$ be the reduced graph of $G$. By Corollary~\ref{nwacreduced} $K$ is $n$--ac and contains no terminal edges.
By Theorem~\ref{no7wac} $K$ has at most two branch points, and by Corollary~\ref{nodegree5n} the degree of each branch point
is at most $4$. We consider the cases when $K$ has no branch points, one branch point or two branch points.

\medskip
\paragraph{{\bf Case 1:}  $K$ has no branch points.} In this case $K$ is either homeomorphic to the arc, $I$, or to the simple closed curve, $S^1$. 

Assume first that $K$ is homeomorphic to $I$, then, by the way $K$ is obtained, $K=G$. Otherwise, reattaching the last 
terminal edge that was removed gives a simple triod which is not $7$--ac, contrary to the hypothesis.
In this case $G$ is on the list.

Next assume $K$ is homeomorphic to $S^1$. If $K=G$, then $G$ is on the list. So assume $G\neq K$, and let $e$
denote the last terminal edge that was removed. Then $K\cup e$ is homeomorphic to the lollipop curve. Furthermore,
$G= K\cup e$, otherwise reattaching the penultimate terminal edge will give a homeomorphic copy of
the graph (a) of Example~(E) which is not $7$--ac, or a simple closed curve with two arcs attached
to it at the same point at one of their end points which is not $7$--ac either. 
Hence, again, $G$ is on the list.

\medskip
\paragraph{{\bf Case 2:}  $K$ has one branch point.} Note that the only possibility for $K$ to have a single
branch point of degree $3$ is for $K$ to be homeomorphic to a simple triod or to the lollipop curve, the
former is not $7$--ac and the latter is not a reduced graph. Hence the degree of the branch point of $K$ is $4$.
In this case $K$ is homeomorphic to either a simple $4$--od, a simple closed curve with two arcs attached
to it at the same point at one of their end points, or to the figure eight curve. The first two cases are not
$7$--ac. Therefore $K$ must be homeomorphic to the figure eight curve. If $G=K$, then $G$ is on the list.
In fact, since attaching an arc to the figure eight curve yields a non $7$--ac curve, we must have that $G=K$.

\medskip
\paragraph{{\bf Case 3:} $K$ has two branch points.} Since the sum of the degrees in a graph is 
always even and $K$ has no terminal edges, then $K$ can not have one branch point of degree $3$ 
and another of degree $4$. Hence the only options are that $K$ has two branch points of 
either degree $3$ or degree $4$.
However, by Corollary~\ref{no24}, $K$ has only two branch points of degree $3$.

If $K$ has two branch points of degree $3$, then it could be homeomorphic to one of the following graphs.

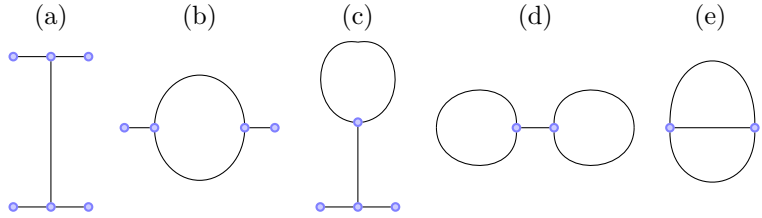
\begin{center}
\begin{tabular}{ccccc}
(a) & (b) & (c) & (d) & (e) \\
\begin{tikzpicture}[vertex/.style={circle,draw=blue!50,fill=blue!20,thick, inner sep=0mm, minimum size=1mm},
  point/.style={coordinate}]

\node[vertex] (base) {};
\node[vertex] (top) at (0,2cm) {};
\node[vertex] (base_l) at (-0.5cm,0) {};
\node[vertex] (base_r) at (0.5cm,0) {};
\node[vertex] (top_l) at (-0.5cm,2cm) {};
\node[vertex] (top_r) at (0.5cm,2cm) {};

\draw (base.north) -- (top.south);
\draw (base_l.east) -- (base.west);
\draw (base.east) -- (base_r.west);
\draw (top_l.east) -- (top.west);
\draw (top.east) -- (top_r.west);

\end{tikzpicture}
&
\begin{tikzpicture}[vertex/.style={circle,draw=blue!50,fill=blue!20,thick, inner sep=0mm, minimum size=1mm},
  point/.style={coordinate}]
\node at (0,-1) {};
\node  at (0,1) {};

\draw (0,0) ellipse (6mm and 7mm);
\node[vertex] (v1) at (180:10mm) {};
\node[vertex] (v2) at (0:10mm) {};
\node[vertex] (v3) at (180:6mm) {};
\node[vertex] (v4) at (0:6mm) {};

\draw (v1) -- (v3);
\draw (v2) -- (v4);
\end{tikzpicture}
&
\begin{tikzpicture} [vertex/.style={circle,draw=blue!50,fill=blue!20,thick, inner sep=0mm, minimum size=1mm}, point/.style={coordinate}]

\node[vertex] (lollipop_bot) {};
\node[vertex] (lollipop_mid) [above=of lollipop_bot] {}; 
\node[point] (pt) [above=of lollipop_mid] {};
\node[vertex] (base_1) at (-0.5cm,0) {};
\node[vertex] (base_2) at (0.5cm,0) {};

\draw  (lollipop_bot.north)  -- (lollipop_mid.south);
\draw  (lollipop_mid.east) to [out=10,in=10,looseness=1.5] (pt);
\draw   (pt) to [out=170, in=170,looseness=1.5] (lollipop_mid.west);
\draw (base_1.east) -- (lollipop_bot.west);
\draw (lollipop_bot.east) -- (base_2.west);
\end{tikzpicture}
&
\begin{tikzpicture}
 [vertex/.style={circle,draw=blue!50,fill=blue!20,thick, inner sep=0mm, minimum size=1mm},
  point/.style={coordinate}]
\node at (0,-1) {};
\node  at (0,1) {};

\node[vertex] (dumbbell_left) {};
\node[vertex] (dumbbell_right) at (0.5cm,0) {}; 
\node[point] (pt1) [left=of dumbbell_left] {};
\node[point] (pt2) [right=of dumbbell_right] {};

\draw  (dumbbell_left.east)  -- (dumbbell_right.west);
\draw  (dumbbell_left.north) to [out=90, in=90,looseness=1.5] (pt1.north);
\draw   (pt1.south) to [out=270, in=270,looseness=1.5] (dumbbell_left.south);
\draw  (dumbbell_right.north) to [out=90, in=90,looseness=1.5] (pt2.north);
\draw   (pt2.south) to [out=270, in=270,looseness=1.5] (dumbbell_right.south);
\end{tikzpicture}
&
\begin{tikzpicture}
 [vertex/.style={circle,draw=blue!50,fill=blue!20,thick, inner sep=0mm, minimum size=1mm},
  point/.style={coordinate}]
\node at (0,-1) {};
\node  at (0,1) {};

\node[vertex] (theta_left) {};
\node[vertex] (theta_right) [right=of theta_left] {};

\draw (theta_left.east) to (theta_right.west) {};
\draw  (theta_left.north) to [out=90, in=90,looseness=2.5] (theta_right.north);
\draw (theta_right.south) to [out=270, in=270,looseness=2] (theta_left.south);
\end{tikzpicture}

\end{tabular}
\end{center}

However the graphs (a), (b), and (c) contain terminal edges. So $K$ can only be homeomorphic 
to the dumbbell (d) or the $\theta$--curve (e); in any case if $G=K$, then $G$ is on the list. 
Note that neither curve, (d) nor (e), can be obtained from a $n$--ac graph ($n\geq 7$) by removing 
a terminal edge since by Theorem~\ref{no7wac} the edge has to be attached to one of the existing branch points; 
it is easy to see that such a graph is not $4$--ac, just take a point on the interior of each edge. Hence $G=K$.
This ends the proof of the theorem.
\end{proof}

\subsection{$\aleph_0$--ac Continua}\label{aleph0}
Call a space  $\kappa$--ac, where $\kappa$ is a cardinal, if every subset of size no more than $\kappa$ is contained in an arc. Note that for finite $\kappa=n$ and $\kappa=\aleph_0$ this coincides with the earlier definitions. For infinite $\kappa$ we have a complete description of $\kappa$--ac continua (not necessarily metrizable), extending Theorem~\ref{main2}. To start let us observe that the arc is $\kappa$--ac for every cardinal $\kappa$. We will see shortly that the arc is the only separable $\kappa$--ac continuum when $\kappa$ is infinite. In particular, the triod and circle are not $\aleph_0$--ac, and so any continuum containing a triod or circle is also not $\aleph_0$--ac. This observation will be used  below.

To state the theorem precisely we need to make a few definitions. Recall that $\omega_1$ is the first uncountable ordinal, or equivalently the set of all countable ordinals, with the induced order topology. Note that a subset of $\omega_1$ is bounded if and only if the set is countable.
The {\it long ray}, $R$, is the lexicographic product of $\omega_1$ with $[0,1)$ with the order topology. We can identify $\omega_1$ (with its usual order topology) with $\omega_1 \times \{0\}$. Evidently $\omega_1$ is cofinal in the long ray. Write $R^-$ for $R$ with each point $x$ relabeled $-x$. The {\it long line}, $L$, is the space obtained by identifying $0$ in  the long ray, $R$, with $-0$ in $R^-$. The topology on the long ray and long line ensures that for any $x < y$ in $R$ (or $L$) the subspace $[x,y]=\{z \in R : x \le z \le y\}$ is (homeomorphic to) an arc. Note that any countable subset of the long ray, or the long line, is bounded, hence both the long ray and long line are $\aleph_0$--ac. To see this for the long ray take any countable subset $S$ then since $\omega_1$ is cofinal in $R$ the set $S$ has an upper bound, $x$ say, and then $S$ is contained in $[0,x]$, which is an arc.

Let $\alpha R$ be the one point compactification of $R$, and $\gamma L$ be the corresponding two point compactification of $L$. The {\it long circle} and {\it long lollipop} are the spaces obtained from $\alpha R$ by identifying the point at infinity to $0$, or any other point, respectively. The {\it long dumbbell}, {\it long figure eight} and {\it long theta} curves come from $\gamma L$ by respectively identifying the negative ($-\infty$) and positive ($+\infty$) endpoints to $-1$ and $+1$, $0$ and $0$, or $+1$ and $-1$. As continuous injective images of the $\aleph_0$--ac spaces $R$ and $L$, all the above spaces are also $\aleph_0$--ac. 

\begin{theorem} \label{aleph0_ch}
Let $K$ be a continuum.
\begin{itemize}
\item[1)] If $K$ is separable and $\aleph_0$--ac then $K$ is an arc.
\item[2)] If $K$ is non--separable, then the following are equivalent:

(i) $K$ is $\aleph_0$--ac, (ii) $K$ is the continuous injective image of a closed sub--interval of the long line, and (iii) $K$ is one of: the long circle, the long lollipop, the long dumbbell, long figure eight,  or the long theta--curve.
\item[3)] If $K$ is $\kappa$--ac for some $\kappa > \aleph_0$, then $K$ is an arc. 
\end{itemize}
\end{theorem}
For part 1) just take a dense countable set, then any arc containing the dense set is the whole space. Part 2) is proved in Proposition~\ref{nonsepaleph0} ((i) $\implies$ (ii)), Proposition~\ref{11image} ((ii) $\implies$ (iii)), while (iii) $\implies$ (i) was observed above with the definition of the curves in 2) (iii). For part 3) note that all non--separable $\aleph_0$--ac spaces (as listed in part 2) (iii)) have a dense set of size $\aleph_1$, and so are not $\aleph_1$--ac. Thus $\kappa$--ac continua for $\kappa \ge \aleph_1$ are separable, hence an arc, by part 1).

It is traditional to use Greek letters ($\alpha, \beta$ etcetera) for ordinals. Consequently we will use the letter `$A$' and variants for arcs, and because in Proposition~\ref{nonsepaleph0} we need to construct a map, in this subsection by an `arc' we mean any homeomorphism between a homeomorph of the closed unit interval and a given space. If $K$ is a space, then by `$A$ is an arc in $K$' we mean the arc $A$ maps into $K$. When $A$ is an arc in a space $K$, then write $\mathop{im}(A)$ for the image of $A$ (it is, of course, a subspace of $K$ homeomorphic to the closed unit interval). For any function $f$, we write $\dom{f}$, for the domain of $f$.

\begin{proposition}{\label{nonsepaleph0}}
Let $K$ be an $\aleph_0$--ac non--separable continuum.  Then there is a continuous bijection $A_\infty : J_\infty\rightarrow K$ where $J_\infty$ is a closed unbounded sub--interval of the long line, $L$. 
\end{proposition}

We prove this by an application of Zorn's Lemma. The following lemmas help to establish that Zorn's Lemma is applicable, and that the maximal object produced is as required. 
\begin{lemma}{\label{l1}} Let $K$ be an $\aleph_0$--ac non--separable continuum.
If $\mathcal{K}$ is a countable collection of separable subspaces of $K$ then there is an arc $A$ in $K$ such that $\bigcup \mathcal{K}\subseteq \mathop{im}(A)$.
\end{lemma}
\begin{proof}
Let $\mathcal{K}=\{S_n : n \in \mathbb{N}\}$ be a countable family of subspaces of $K$, and, for each $n$,  let $D_n$ be a countable dense subset of $S_n$. Let $D=\bigcup_n D_n$ --- it is countable. Since $K$ is $\aleph_0$--ac there is an arc $A$ in $K$ such that $D \subseteq \mathop{im}(A)$. As $D$ is dense in $\bigcup \mathcal{K}$ and $\mathop{im}(A)$ is closed, we see that $\bigcup \mathcal{K}  \subseteq \mathop{im}(A)$. 
\end{proof}

\begin{lemma}{\label{l3}} Let $K$ be an $\aleph_0$--ac non--separable continuum. Suppose $[a,b]$ is a proper closed subinterval of $L$ (or $R$), $A : [a,b] \to K$ is an arc in $K$ and $y\in K\setminus \mathop{im}(A)$.  Then either (i) for every $c>b$  in $L$ there is an arc $A':[a,c]\rightarrow K$ such that $A'\restriction_{[a,b]} =A$ and $A'(c)=y$, or (ii) for every $c<a$  in $L$ there is an arc $A':[c,b]\rightarrow K$ such that $A'\restriction_{[a,b]} =A$ and $A'(c)=y$.   
\end{lemma}
\begin{proof} Fix $a,b$, the arc $A$ and $y$.
Let $\mathcal{K}=\{\mathop{im}(A) , \{y\}\}$, and apply Lemma~\ref{l1} to get an arc $A_0:[0,1] \to K$ in $K$ such that $\mathop{im}(A_0) \supseteq \mathop{im}(A) \cup \{y\}$. Let $J=A_0^{-1} (\mathop{im}(A))$,  $a'=\min J$, $b'=\max J$ and $c'=A_0^{-1} (y)$. Without loss of generality (replacing $A_0$ with $A_0 \circ \rho$ where $\rho (t)=1-t$ if necessary) we can suppose that $A_0(a')=A(a)$ and $A_0(b')=b$.

Since $y\not\in \mathop{im}(A)$, either $c'>b'$ or $c'<a'$. Let us suppose that $c' >b'$. This will lead to case (i) in the statement of the lemma. The other choice will give, by a very similar argument which we omit, case (ii). Take any $c$ in $L$ such that $c>b$. Let $A_1$ be a homeomorphism of the closed subinterval $[a,c]$ of $L$ with the subinterval $[a',c']$ of $[0,1]$ such that $A_1(a)=a'$, $A_1(b)=b'$ and $A_1(c)=c'$. Set $A_2=A_0 \circ A_1 : [a,c] \to K$. So $A_2$ is an arc in $K$ such that $A_2(a)=A(a)$, $A_2(b)=A(b)$,  $A_2(c)=y$ and $A_2([a,b])= \mathop{im} (A)$. The arc $A_2$ is almost what we require for $A'$ but it may traverse the (set) arc $\mathop{im}(A)$ at a `different speed' than $A$. Thus we define $A':[a,c] \to K$ to be equal to $A$ on $[a,b]$ and equal to $A_2$ on $[b,c]$. Then $A'$ is the required arc. 
\end{proof}

\begin{proof}{\bf (Of Proposition~\ref{nonsepaleph0})} 
Let $\mathcal{A}$ be the set of all continuous injective maps $A: J \to K$ where $J$ is a closed subinterval of $L$, ordered by: $A \le A'$ if and only if $\dom{A} \subseteq \dom{A'}$ and $A'\restriction_{\dom{A}} = A$. Then $\mathcal{A}$ is the set of all candidates for the map we seek, $A_\infty$. 
We will apply Zorn's Lemma to $(\mathcal{A},\le)$ to extract $A_\infty$. To do so we need to verify that $(\mathcal{A},\le)$ is non--empty, and all non--empty chains have upper bounds.

As $K$ is $\aleph_0$--arc connected we know there are many arcs in $K$, so the set $\mathcal{A}$ is not empty.
Now take any non--empty chain $\mathcal{C}$ in $\mathcal{A}$. We show that $\mathcal{C}$ has an upper bound.
Let $\mathcal{J}=\{ \dom{A'} : A' \in \mathcal{C}\}$. Since $\mathcal{J}$ is a chain of subintervals in $L$, the set $J = \bigcup \mathcal{J}$ is also a subinterval of $L$. Define $A:J \to K$ by $A(x)=A'(x)$ for any $A'$ in $\mathcal{C}$ with $x \in \dom{A'}$. Since $\mathcal{C}$ is a chain of injections, $A$ is well--defined and injective. Since the domains of the functions in $\mathcal{C}$ form a chain of subintervals, any point $x$ in $J$ is in the $J$--interior of some $\dom{A'}$ (there is a set $U$, open in $J$ such that $x \in U \subseteq \dom{A'}$), where $A' \in \mathcal{C}$, and so $A$ coincides with $A'$ on some $J$--neighborhood of $x$, thus, since $A'$ is continuous at $x$, the map $A$ is also continuous at $x$. If $J$ is closed, then we are done: $A$ is in $\mathcal{A}$ and $A \ge A'$ for all $A'$ in $\mathcal{C}$.

If the interval $J$ is not closed then it has at least one endpoint (in $L$) not in $J$. We will suppose $J=(a,\infty)$. The other cases, $J=(a,b)$ and $J=(-\infty,a)$, can be dealt with similarly. We show that we can continuously extend $A$ to $[a,\infty)$. If so then $A$ will be injective, hence in $\mathcal{A}$, and  an upper bound for $\mathcal{C}$. Indeed the only way the extended $A$ could fail to be injective was if $A(a)=A(c)$ for some $c>a$, and then $A([a,c])$ is a circle in $K$, contradicting the fact that $K$ is $\aleph_0$--ac.

Evidently it suffices to continuously extend $A'=A\restriction _{(a,b]}$ to $[a,b]$.
Let $\mathcal{K}=\{ A( (a,b])\}$ and apply Lemma~\ref{l3} to see that $A'$ maps the half open interval, $(a,b]$, into (a homeomorphic copy of) the closed unit interval. So we can apply some basic real analysis to get the extension. Indeed the map $A'$ is continuous and injective, and hence strictly monotone. By the inverse function theorem, $A'$ has a continuous inverse, and so is a homeomorphism  of $(a,b]$ with some half open interval,  $(c,d]$ or $[d,c)$ in the closed unit interval. Defining $A(a)=c$ gives the desired continuous extension.

Let $A_\infty$ be a maximal element of $\mathcal{A}$. Then its domain is a closed subinterval of the long line, $L$. 
We first check that $\dom{A_\infty}$ is not bounded. Then we prove that $A_\infty$ maps onto $K$.

If $A_\infty$ had a bounded domain, then it is an arc. So it has separable image. As $K$ is not separable we can pick a point $y$ in $K \setminus \mathop{im}(A_\infty)$. Applying Lemma~\ref{l3} we can properly extend $A_\infty$ to an arc $A'$. But then $A'$ is in $\mathcal{A}$, $A_\infty \le A'$ and $A_\infty \ne A'$, contradicting maximality of $A_\infty$.

We complete the proof by showing that $A_\infty$ is surjective. We go for a contradiction and suppose that instead there is  a point $y$ in $K \setminus \mathop{im}(A_\infty)$. Two cases arise  depending on the domain of $A_\infty$. 

Suppose first that $\dom{A_\infty}=L$.  Pick a point $x$ in $\mathop{im}(A_\infty)$. Pick an arc $A$ from $x$ to $y$. Taking a subarc, if necessary, we can suppose $A:[0,1] \to K$, $A(0)=x$ and $A(t) \notin \mathop{im}(A_\infty)$ for all $t >0$. Let $x'=A_{\infty}^{-1}(x)$. Pick any $a',b'$ from $L$ such that $a' < x' < b'$. Then the subspace $A_\infty ([a',b']) \cup A([0,1])$ is a triod in $K$, which contradicts $K$ being $\aleph_0$--sac.

Now suppose that $\dom{A_\infty}$ is a proper subset of $L$. Let us assume that $\dom{A_\infty}=[a,\infty)$. (The other case, $\dom{A_\infty}=(-\infty,a]$, follows similarly.) Pick any $b>a$, and apply Lemma~\ref{l3} to $A=A_{\infty}\restriction_{[a,b]}$ and $y$. If case (ii) holds then pick any $c<a$ and $A$ can be extended `to the left' to an arc $A'$ with domain $[c,b]$. This gives a proper extension of $A_\infty$ defined on $[c,\infty)$ (which is $A'$ on $[c,a]$ and $A_\infty$ on $[a,\infty)$), contradicting maximality of $A_\infty$.

So case (i) must hold. Pick any $c >b$, and we get an arc $A':[a,c] \to K$ in $K$ extending $A$. Let $T=A_{\infty} ([a,c]) \cup A'([a,c])$. Observe that $T$ has at least three non cutpoints, namely $A'(a)=A_{\infty}(a)$, $A_\infty(c)$ and $A'(c)$. So $T$ is not an arc, but it is a separable subcontinuum of the $\aleph_0$--ac continuum $K$, which is the desired contradiction.
\end{proof}

To complete the proof of Theorem~\ref{aleph0_ch} it remains to identify the continuous injective images of closed sub--intervals of the long line. We recall some basic definitions and facts connected with the space of countable ordinals, $\omega_1$ (see \cite{Kunen}, for example). A subset of $\omega_1$ is \emph{closed and unbounded} if it is cofinal in $\omega_1$ and closed in the order topology. A countable intersection of closed and unbounded sets is closed and unbounded. The set $\Lambda$ of all limit ordinals in $\omega_1$ is a closed and unbounded set. A subset of $\omega_1$ is \emph{non-stationary} if it is contained in the complement of a closed and unbounded set. A subset of $\omega_1$ is \emph{stationary} if it is not non-stationary, or equivalently if it meets every closed and unbounded set. The Pressing Down Lemma (also known as Fodor's lemma) states than if $S$ is a stationary set and $f: S \to \omega_1$ is regressive (for every $\alpha$ in $S$ we have $f(\alpha) < \alpha$) then there is a $\beta$ in $\omega_1$ such that $f^{-1}(\beta)$ is cofinal in $\omega_1$.
  
\begin{proposition}\label{11image}
If $K$ is a non--separable continuum and is the continuous injective image of a closed sub--interval of the long line, then $K$ is one of: the long circle, the long lollipop, the long dumbbell, long figure eight,  or the long theta--curve.
\end{proposition}

\begin{proof} The closed non--separable sub--intervals of the long line are (up to homeomorphism) just the long ray and long line, itself. 

Let us suppose for the moment that the $K$ is the continuous injective image of the long ray, $R$. We may as well identify points of $K$ with points in $R$. Note that on  any closed subinterval, $[a,b]$ say, of $R$, (by compactness of $[a,b]$ in $R$, and Hausdorffness of $K$) the standard order topology and the $K$--topology coincide. It follows that at any point with a bounded $K$--open neighborhood the standard topology and $K$--topology agree.
We will show that there is a point $x$ in $R$ such that every $K$--open $U$ containing $x$ contains a tail, $(t,\infty)$, for some $t$. Assuming this, then by Hausdorffness of $K$, every  point distinct from $x$ has bounded neighborhoods, and so $x$ is the only point where the $K$--topology differs from the usual topology. Then $K$ is either the long circle or long dumbbell depending on where $x$ is in $R$ (in particular, if it equals $0$). The corresponding result for continuous injective images of the long line follows immediately.

Suppose, for a contradiction, that for every $x$ in $R$, there is a $K$--open set $U_x$ containing $x$ such that $U_x$ contains no tail. By compactness of $K$, some finite collection, $U_{x_1}, \ldots , U_{x_n}$, covers $K$. Let $S_{i}=U_{x_i} \cap \Lambda$, where $\Lambda$ is the set of limits in $\omega_1$. Then (since the finitely many $S_i$ cover the closed unbounded set $\Lambda$) at least one of the $S_i$ is stationary. Take any $\alpha$ in $S_i$, and consider it as a point of the closed subinterval $[0,\alpha]$ of $R$, where we know the standard topology and the $K$--topology agree. Since $\alpha$ is a limit point which is in $U_{x_i} \cap [0,\alpha]$, and this latter set is open, we know there is ordinal $f(\alpha)<\alpha$ such that $(f(\alpha),\alpha] \subseteq U_{x_i}$. Thus we have a regressive map, $f$, defined on the stationary set $S_i$, so by the Pressing Down Lemma there is a $\beta$ such that $f^{-1} (\beta)$ is cofinal in $\omega_1$. Hence $U_{x_i}$ contains $\bigcup \{ (f(\alpha),\alpha] :  \alpha \in f^{-1}(\beta)\} = (\beta,\infty)$, and so $U_{x_i}$ does indeed contain a tail.  
\end{proof}

\section{Complexity of Characterizations}\label{cplxty}
Theorem~\ref{main3} from the Introduction makes certain claims about the complexity of characterizing, for various $n$, the $n$--ac graphs which are not $(n+1)$--ac. We introduce the necessary technology from descriptive set theory to make these claims precise. Then Theorem~\ref{main3formal} is the formalized version of Theorem~\ref{main3}. 

 Recall (see \cite{kech}) that the Borel subsets of a space ramify into a hierarchy, $\Pi_\alpha, \Sigma_\alpha$, indexed by countable ordinals. Sets lower in the hierarchy are less complex than those found higher up.  Most relevant here are: $\Pi_3$ which is the set of $F_{\sigma \delta}$ subsets, $\Sigma_3$ which is the set of all $G_{\delta \sigma}$ subsets, and  $D_2(\Sigma_3)$ the set of intersections of one $\Pi_3$ and one $\Sigma_3$ set.  

The complexity of a set in terms of its position in the Borel hierarchy is precisely correlated to the complexity of the logical formulae needed to define it. A $\Pi_3$ set, $S$, can be defined by a formula, $\phi$ (via $S=\{ x : \phi (x) \text{ is true}\}$), of the form $\forall p \exists q \forall r \ \text{(something simple)}$, where the quantifiers run over {\it countable} sets, and `something simple' is boolean.  A $\Sigma_3$ set, $T$, can be defined by a formula, $\psi$, of the form $\exists p \forall q \exists r \ \text{(something simple)}$. While a $D_2(\Sigma_3)$ set can be defined by a formula of the form $\phi \land \psi$, where $\phi$ and $\psi$ are as above. 

For example, let $S_3^* = \{ \alpha \in 2^{\mathbb{N} \times \mathbb{N}} : \exists J \, \forall j >J \, \exists k \ \alpha (j,k)=0\}$, and $P_3 = \{ \beta \in 2^{\mathbb{N} \times \mathbb{N}} : \forall j \, \exists K \, \forall k \ge K \ \beta (j,k) =0\}$. Then $S_3^*$ is $\Sigma_3$, and $P_3$ is $\Pi_3$ in $2^{\mathbb{N} \times \mathbb{N}}$.  And $S_3^* \times P_3$ is a $D_2(\Sigma_3)$ subset of $\left( 2^{\mathbb{N} \times \mathbb{N}} \right)^2$.

For a class of subsets $\Gamma$, a set $A$ is $\Gamma$--hard if $A$ is not in any proper subclass, while it is $\Gamma$--complete if it $\Gamma$--hard and in $\Gamma$. In other words, $A$ is $\Gamma$--complete if and only if it has complexity precisely $\Gamma$. It is known \cite{kech} that $S_3^*$ is $\Sigma_3$--complete, $P_3$ is $\Pi_3$--complete, and $S_3^* \times P_3$ is $D_2(\Sigma_3)$--complete. We can re--phrase these last two statements as follows: there is a formula characterizing $P_3$ of the form, $\forall \exists \forall$ but we can be {\sl certain} that {\it no logically simpler characterizing formula exists}, and there is a  formula characterizing $S^*_3 \times P_3$ of the form, $(\exists \forall \exists) \land (\forall \exists \forall)$ but  no logically simpler characterizing formula exists.

Let $A \subseteq X$, $B \subseteq Y$ and $f$ a continuous map of $X$ to $Y$ such that $f^{-1}(B)=A$ (such an $f$ is a {\it Wadge reduction}). Note that if $B$ is in some Borel class $\Gamma$, then by continuity so is $A=f^{-1}(B)$. Hence if $A$ is $\Gamma$--hard, then so is $B$.

We work inside the hyperspace $C(I^N)$ of all subcontinua of $I^N$ with the Vietoris topology, which makes it a continuum. In light of the remarks above, it should now be clear that the following is indeed a formal version of Theorem~\ref{main3}.
\begin{theorem}{\label{main3formal}} \ 
Fix $N \ge 2$. Inside the space $C(I^N)$:
\begin{itemize}
\item[1)] the set of graphs which are $\omega$--ac is $\mathbf{\Pi}_3$--complete,

\item[2)] any family of homeomorphism classes of graphs 
 is $\Pi_3$--hard and always  $D_2(\mathbf{\Sigma}_3)$, and 

\item[3)] the set of $n$--ac  not $(n+1)$--ac graphs is $D_2(\mathbf{\Sigma}_3)$--complete, for $n=2,3,4,5$. 
\end{itemize}
\end{theorem}
Claims 1)--3) are the contents of Lemmas~\ref{omega_c}, \ref{between} and Proposition~\ref{n_c}, respectively.

\begin{lemma}\label{between}
Let $\mathcal{C}$ be any collection of graphs. Then $H(\mathcal{C})$, the set of all subcontinua of $I^N$ homeomorphic to some member of $\mathcal{C}$, is $\Pi_3$--hard and in $D_2(\Sigma_3)$. 
\end{lemma}
 
\begin{proof} That $H(\mathcal{C})$ is $\Pi_3$--hard is immediate from Theorem~7.3 of \cite{cdm}. It remains to show it is in $D_2(\Sigma_3)$.

For spaces $X$ and $Y$, write $X \le Y$ if $X$ is $Y$--like,  $X < Y$ if $X \le Y$ but $Y \not \le X$ and $X \sim Y$ if $X \le Y$ and $Y \le X$. Further write $\mathcal{L}_X = \{ Y : Y \le Y\}$ and $Q(X)=\{ Y : X \sim Y\}$.

Let $\mathcal{C}_0$ be a maximal family of pairwise nonhomeomorphic members of $\mathcal{C}$.
Up to homeomorphism there are only countably many graphs. So enumerate $\mathcal{C}_0=\{G_m : m \in \mathbb{N}\}$. According to Theorem~1.7 of \cite{cdm}, for a graph $G$ and Peano continuum, $P$,  we have that $P$ is $G$--like if and only if $P$ is a graph obtained from $G$ by identifying to  points disjoint (connected) subgraphs. For a fixed graph $G$, then, there are, up to homeomorphism, only finitely many $G$--like graphs. For each $G_m$ in $\mathcal{C}$ pick graphs $G_{m,i}$ for $i=1, \ldots , k_m$ such that each $G_{m,i}$ is $<G$, and if $G'$ is a graph such that $G' < G$ then for some $i$ we have $H(G')=H(G_{m,i})$.

For a graph $G$, $H(G)=Q(G)$ (\cite{KaYe}). Hence, writing $\mathcal{P}$, for the class of Peano continua, we have that $H(\mathcal{C}) = \bigcup_m Q(G_m) = \mathcal{P} \cap \left( \bigcup_m R_{m} \right)$, where $R_{m}=\mathcal{L}_{G_m} \setminus \bigcup_{i=1}^{k_m} \mathcal{L}_{G_{m,i}} = \mathcal{L}_{G_m} \cap \left( C(I^N) \setminus \bigcup_{i=1}^{k_m} \mathcal{L}_{G_{m,i}}\right)$.

By Corollary~5.4 of \cite{cdm}, for a graph $G$, the set $\mathcal{L}_G$ is $\Pi_2$. Hence each $R_{m}$, as the intersection of a $\Pi_2$ and a $\Sigma_2$, is $\Sigma_3$, and so is their countable union. Since $\mathcal{P}$ is $\Pi_3$, we see that $H(\mathcal{C})$ is indeed the intersection of a $\Pi_3$ set and a $\Sigma_3$ set. 
\end{proof}

\begin{lemma}\label{omega_c} The set $AC_\omega$ of all subcontinua of $I^N$ which are $\omega$--ac graphs is $\Pi_3$--complete.
\end{lemma}

\begin{proof} For a graph $G$, $H(G)$ is $\Pi_3$. By Theorem~\ref{main1}, $AC_{\omega}$ is a finite union of $H(G)$ for graphs $G$, and so is also $\Pi_3$. Hence by Lemma~\ref{between}, $AC_\omega$ is $\Pi_3$--complete.
\end{proof}

\begin{proposition}\label{n_c}
For any $n$, let $AC_n$ be the set of subcontinua of $I^N$ which are $n$--ac but not $(n+1)$--ac graphs. Then for $n=2,3,4$ and $5$ the sets $AC_n$ are $D_2(\Sigma_3)$--complete.
\end{proposition}

\begin{proof} According to Lemma~\ref{between} $AC_n$ is  $D_2(\Sigma_3)$, so it suffices to show that $AC_n$ is $D_2(\Sigma_3)$--hard.

To show that $AC_n$ is $D_2(\Sigma_3)$--hard it suffices to show that there is a continuous map $F: \left( 2^{\mathbb{N}\times \mathbb{N}} \right)^2  \to C(I^N)$ such that $F^{-1} (AC_n) = S_3^* \times P_3$. We do the construction for $N=2$. Since $\mathbb{R}^2$ embeds naturally in general $\mathbb{R}^N$, the proof obviously extends to all $N\ge2$.

We first do the case when $n=5$. Then we will explain how to make the minor modifications needed for the other cases, $n=2,3$ and $4$.

For $x, y$ in $\mathbb{R}^2$, let $\jn{x y}$ be the straight line segment from $x$ to $y$. Set $O=(0,0)$, $T=(3,1)$, $B_1=(1,0)$, $B_2=(4/3,0)$, $B_3=(5/3,0)$, $B_4=(2,0)$ and $T_1=(1,1)$, $T_2=(4/3,1)$, $T_3=(5/3,1)$, $T_4=(2,1)$. Let $K_0=\jn{O B_4} \cup \jn{B_4 T} \cup \jn{T T_1} \cup \jn{T_1 O} \cup \jn{B_2 T_2} \cup \jn{B_3 T_3}$. Then $K_0$ is a $5$--ac not $6$--ac graph. Define $b_j=(1/j,0)$, $t_j=(1/j,1/j)$, $t_j^k=(1/j,1/j-1/(kj))$ and $s_j^k=(1/j-1/(kj(j+1)),0)$. Then  $K_J = K_0 \cup \bigcup_{j=1}^J \jn{b_j t_j}$ --- for each $J$ --- is also a $5$--ac not $6$--ac graph. 

 Let $K_0'$ be $K_0$ with the interior of the line from $O$ to $B_1$, and the interior of the line from $T_4$ to $T$, deleted.

We now define $F$ at some $\alpha$ and $\beta$ in $2^{\mathbb{N} \times \mathbb{N}}$. Fix $j$. 
If $\alpha (j,k)=1$ for all $k$, then let $R_j=\jn{b_j t_j} \cup \jn{b_j b_{j+1}}$. Otherwise, let $k_0=\min \{ k : \alpha (j,k)=0\}$, and let $R_j = \jn{b_j t_j^{k_0}} \cup \jn{t_j^{k_0} s_j^{k_0}} \cup \jn{s_j^{k_0} b_{j+1}}$. 

For any $j,k$ set $p_j=3-1/j$, $q_j^k=1-1/(j+k)$, $\ell_j=p_j + (1/8)(p_{j+1} - p_j)$ and $r_j=p_j + (7/8)(p_{j+1}-p_j)$. Fix $j$. Define
\begin{eqnarray*}
S_j &=&  
\jn{ (p_j,1) (p_j, q_j^1)} \cup \jn{(p_j,q_j^1) (\ell_j,q_j^1)} \cup \jn{(\ell_j,1) (p_{j+1},1)}\\ 
&\cup &  \bigcup \{  \jn{(\ell_j,q_j^k) (\ell_j,q_j^{k+1}) }    : \beta (j,k)=0\} \\
& \cup &  \bigcup \{ \jn{(\ell_j,q_j^k) (r_j,q_j^k)} \cup \jn{(r_j,q_j^k) (\ell_j,q_j^{k+1})}     : \beta (j,k)=1\}.
\end{eqnarray*}

Let $F(\alpha,\beta) = K_0' \cup \bigcup_j (R_j \cup S_j)$. Then it is straightforward to check $F$ maps $\left(2^{\mathbb{N} \times \mathbb{N}}\right)^2$ continuously into $C([0,4]^2)$.

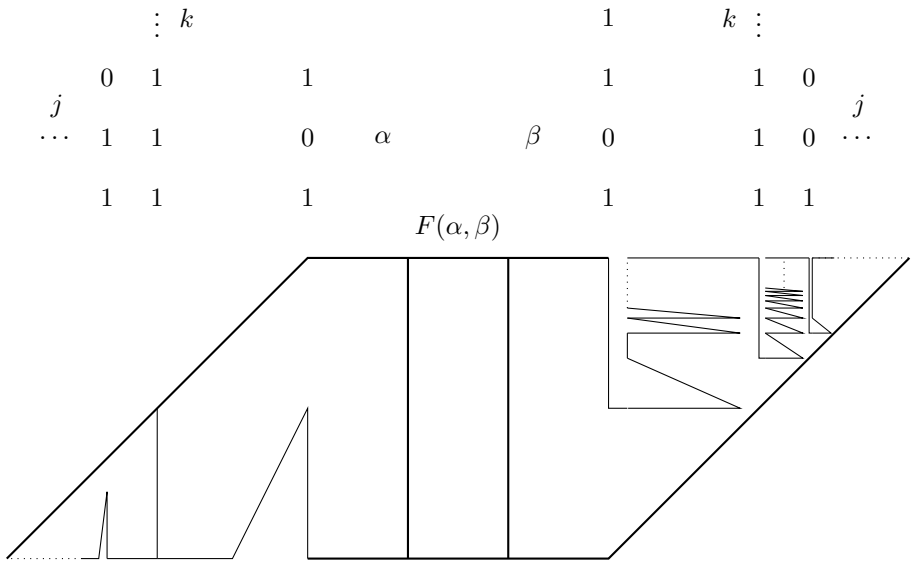
\begin{center}
\begin{tikzpicture}[scale=4]
\draw[thick] (1,0) -- (2,0) -- (3,1);
\draw[thick] (2,1) -- (1,1) -- (0,0);
\draw[thick] (4/3,0) -- (4/3,1);
\draw[thick] (5/3,0) -- (5/3,1);

\draw (1,0) -- (1, 1-1/2) -- (1-1/4,0) -- (1/2,0);

\node at (1,1.2) {1};
\node at (1,1.4) {0};
\node at (1,1.6) {1};

\draw (1/2,0) -- (1/2,1/2);
\draw (1/2,0)-- (1/3,0);
\node at (1/2,1.2) {1};
\node at (1/2,1.4) {1};
\node at (1/2,1.6) {1};

\draw (1/3,0) -- (1/3, 1/3-1/9) -- (1/3-1/36,0) -- (1/4,0);
\node at (1/3,1.2) {1};
\node at (1/3,1.4) {1};
\node at (1/3,1.6) {0};


\draw [dotted] (1/4,0) -- (0,0);
\node at (1/6,1.4) [label=above:$j$] {\ldots};
\node at (1/2,1.8) [label=right:$k$] {\vdots};

\node at (1.25,1.4) {$\alpha$};


\node at (1.75,1.4) {$\beta$};

\draw (2,1) -- (2,0.5) -- (2.0625, 0.5);
\draw (2.0625,1) -- (2.5,1);
\draw (2.065,0.5) -- (2.4375,0.5) -- (2.0625,0.6666);
\draw (2.0625,0.6666) -- (2.0625,0.75);
\draw (2.0625,0.75) -- (2.4375,0.75) -- (2.0625,0.8); 
\draw (2.0625,0.8) -- (2.4375,0.8) -- (2.0625,0.83333); 

\draw[dotted] (2.0625,0.83333) -- (2.0625,1);

\node at (2,1.2) {1};
\node at (2,1.4) {0};
\node at (2,1.6) {1};
\node at (2,1.8) {1};

\draw (2.5,1) -- (2.5,0.6666) -- (2.5208325, 0.6666);
\draw (2.5208325,1) -- (2.6666,1);
\draw (2.5208325,0.6666) -- (2.6458275,0.6666) -- (2.5208325,0.75);
\draw (2.5208325,0.75) -- (2.6458275,0.75) -- (2.5208325,0.8);
\draw (2.5208325,0.8) -- (2.6458275,0.8) -- (2.5208325,0.83333); 
\draw (2.5208325,0.83333) -- (2.6458275,0.83333) -- (2.5208325,0.857143); 
\draw (2.5208325,0.857143) -- (2.6458275,0.857143) -- (2.5208325,0.875); 
\draw (2.5208325,0.875) -- (2.6458275,0.875) -- (2.5208325,0.88888); 
\draw (2.5208325,0.88888) -- (2.6458275,0.88888) -- (2.5208325,0.9); 

\draw[dotted] (2.58333,0.9) -- (2.58333,1);

\node at (2.5,1.2) {1};
\node at (2.5,1.4) {1};
\node at (2.5,1.6) {1};
\node at (2.5,1.8) [label=left:$k$] {\vdots};

\draw[very thin] (2.6666,1) -- (2.6666,0.75)--(2.6770775, 0.75);
\draw (2.6770775,1) -- (2.75,1);
\draw[very thin] (2.6770775,0.75) -- (2.7395825,.75) -- (2.6770775,0.8);
\draw[very thin] (2.6770775,0.8) -- (2.6770775,1);

\node at (2.6666,1.2) {1};
\node at (2.6666,1.4) {0};
\node at (2.6666,1.6) {0};

\draw[dotted] (2.6770775,1) -- (3,1);
\node at (2.83333,1.4) [label=above:$j$] {\ldots};

\node at (1.5,1.1) {$F(\alpha, \beta)$};

\end{tikzpicture}
\end{center}

Take any $\alpha$. For any $j$, the set $R_j$ connects the bottom edge $\jn{O B_1}$ with the diagonal edge $\jn{O T_1}$ if  $\alpha(j,k)=1$ for all $k$, and otherwise is an arc from $b_j$ to $b_{j+1}$. Hence $\bigcup_{j>J} R_j$ is a free arc from $B_{J+1}$ to $O$ if $\alpha$ is in $S_3^*$, and otherwise can't be a subspace of a graph (because it contains infinitely many points of order $3$).

Take any $\beta$. For any $j$, $S_j$ is an arc from $(p_j,1)$ to $(p_{j+1},1)$ if $\beta(j,k)=0$ for all but finitely many $k$, but contains a `topologists sine curve' if $\beta(j,k)=1$ for infinitely many $k$. Thus $\bigcup_j S_j$ is a free arc from $T_4$  to $T$ if $\beta$ is in $P_3$, and otherwise can't be a subspace of a graph (because it contains a `topologists sine curve'). 

Hence if $(\alpha,\beta)$ is in $S_3^* \times P_3$,  $F(\alpha,\beta)$ is homeomorphic to some $K_J$, which in turn means it is a graph which is $5$--ac but not $6$--ac. On the other hand, if either $\alpha$ is not in $S_3^*$ or $\beta$ is not in $P_3$, then $F(\alpha,\beta)$ contains subspaces which can't be subspaces of a graph --- and so is not a graph. Thus $F^{-1} (AC_5) = S_3^* \times P_3$ as required.

\smallskip
Let $T_1^+=(1,2)$, $T_2^+=(4/3,2)$ and $T_3^+=(5/3,2)$.
Suppose now that $n=4$. Modify $K_0$ by adding a line segment from $T_1$ to $T_1^+$. Then this modified $K_0$ is $4$--ac but not $3$--ac. Further, for any $J$, the modified $K_J$ obtained by taking the modified $K_0$ as a base is also $4$--ac but not $5$--ac. Thus we get the desired reduction in the case when $n=4$.

Similarly, for $n=3$, modify $K_0$ by adding both a line segment $\jn{T_1 T_1^+}$ and $\jn{T_2 T_2^+}$. This gives a base graph, and family of $K_J$, which are all $3$-ac but not $4$--ac. Finally, by adding the three line segments,  $\jn{T_1 T_1^+}$, $\jn{T_2 T_2^+}$ and $\jn{T_3 T_3^+}$ to $K_0$ we get $2$--ac not $3$--ac graphs. The  desired reductions for $n=3$ and $n=2$ follow.
\end{proof}

\section{Open Problems} \label{probs}

The main theorems, Theorem~\ref{main1}, \ref{main2} and~\ref{main3}, raise some natural problems.

\begin{problems} \ 

\begin{itemize}
\item Find examples of continua which are $n$--ac but not $(n+1)$--ac for $n \ge 7$. 

Theorem~\ref{main1} implies that no graph is an example. Are there regular examples?

\item Find a `simple' (i.e. $\Pi_3$) characterization of $6$--ac graphs which are not $7$--ac. Alternatively,  prove that no such characterization is possible, and show that the set of $6$--ac not $7$--ac graphs is $D_2(\Sigma_3)$.

Note (Examples~\ref{exs} (G)) that there are infinitely many  $6$--ac not $7$--ac graphs --- rather than the only finite family of $7$--ac graphs --- but this does not rule out a `simple' characterization.

\item Characterize the $\omega$--ac regular continua.

The Sierpinski triangle is a regular $\omega$--ac continuum. The authors, with Kovan--Bakan, show in \cite{egkm} that there is no Borel characterization of {\it rational} $\omega$--ac continua. However the examples used in that argument are far from regular (not even locally connected).
\end{itemize}
\end{problems}

\section*{Acknowledgements} The authors thank the referee for his/her comments that improved the
 paper and for observing that Kuratowski's $K_{3,3}$ graph is $6$-ac but not $7$-ac
 (see Example \ref{exs} (F)).

\end{document}